\documentclass[a4paper,10pt,leqno]{amsart}
\title[$L^2$-Betti number and approximation in  arbitrary characteristic]
{The first $L^2$-Betti number and approximation in arbitrary characteristic}
\author{Mikhail Ershov}
\thanks{The first author is supported in part by the
NSF grant DMS-0901703 and  the Sloan Research Fellowship grant BR 2011-105.}
\address{Department of Mathematics, University of Virginia, PO Box 400137, Charlottesville, VA 22904-4137, U.S.A.}
\email{ershov@virginia.edu}
\urladdr{http://people.virginia.edu/~mve2x}
\author{Wolfgang L\"uck}
\thanks{This paper is financially supported by the Leibniz-Preis of the second author. }
        \address{Mathematisches Institut der Universit\"at Bonn, Endenicher Allee 60, 53115 Bonn, Germany}
         \email{wolfgang.lueck@him.uni-bonn.de}
          \urladdr{http://www.him.uni-bonn.de/lueck}
         \date{September 23,  2012}
\keywords{First $L^2$-Betti number, approximation in prime characteristic}
\subjclass[2010]{Primary: 20F65; Secondary: 46Lxx}

\usepackage{hyperref}
\usepackage{color}
\usepackage{pdfsync}
\usepackage{calc}
\usepackage{enumerate,amssymb}
\usepackage[arrow,curve,matrix,tips,2cell]{xy}
  \SelectTips{eu}{10} \UseTips
  \UseAllTwocells

\DeclareMathAlphabet\EuR{U}{eur}{m}{n}
\SetMathAlphabet\EuR{bold}{U}{eur}{b}{n}

\makeindex             

\theoremstyle{plain}
\newtheorem{theorem}{Theorem}[section]
\newtheorem{lemma}[theorem]{Lemma}

\newtheorem{corollary}[theorem]{Corollary}

\theoremstyle{definition}

\newtheorem{definition}[theorem]{Definition}

\newtheorem{remark}[theorem]{Remark}
\newtheorem{question}[theorem]{Question}

{\catcode`@=11\global\let\c@equation=\c@theorem}

\newcommand{\comsquare}[8]                   
{\begin{CD}
#1 @>#2>> #3\\
@V{#4}VV @V{#5}VV\\
#6 @>#7>> #8
\end{CD}
}

\newcommand{\xycomsquare}[8]                   
{\xymatrix
{#1 \ar[r]^{#2} \ar[d]^{#4} &
#3 \ar[d]^{#5}  \\
#6\ar[r]^{#7} &
#8
}
}

\newcommand{\xycomsquareminus}[8]                      
{\xymatrix{#1 \ar[r]^-{#2} \ar[d]^-{#4} &
#3 \ar[d]^-{#5}  \\
#6\ar[r]^-{#7} &
#8
}
}

\newcommand{\caln}{{\mathcal N}}

\newcommand{\IF}{{\mathbb F}}

\newcommand{\IN}{{\mathbb N}}

\newcommand{\IQ}{{\mathbb Q}}
\newcommand{\IR}{{\mathbb R}}

\newcommand{\IZ}{{\mathbb Z}}

\newcommand{\id}{\operatorname{id}}

\newcommand{\im}{\operatorname{im}}

\newcommand{\pr}{\operatorname{pr}}

\newcommand{\RG}{\operatorname{RG}}

\newcommand{\tors}{\operatorname{tors}}

\newcommand \la {\langle}
\newcommand \ra {\rangle}

\newcommand{\higherlim}[3]{{\setbox1=\hbox{\rm lim}
        \setbox2=\hbox to \wd1{\leftarrowfill} \ht2=0pt \dp2=-1pt
        \mathop{\vtop{\baselineskip=5pt\box1\box2}}
        _{#1}}^{#2}#3}

\newcommand{\version}[1]                       
{\begin{center} last edited on #1\\
last compiled on \today\\
name of texfile: \jobname
\end{center}
}

\newcounter{commentcounter}

\begin{document}


\typeout{------------------------- Abstract ------------------------------------}

\begin{abstract}
Let $G$ be a finitely generated group and $G = G_0 \supseteq G_1
  \supseteq G_2 \supseteq \cdots$ a descending chain of finite index normal
subgroups of $G$. Given a field $K$, we consider the sequence
$\frac{b_1(G_i;K)}{[G:G_i]}$ of normalized first Betti numbers
of $G_i$ with coefficients in $K$, which we call a $K$-approximation
for $b_1^{(2)}(G)$, the first $L^2$-Betti number of $G$. In this
paper we address the questions of when $\IQ$-approximation
and $\IF_p$-approximation have a limit, when these limits coincide,
when they are independent of the sequence $(G_i)$ and how they are related
to $b_1^{(2)}(G)$. In particular, we prove the inequality $\lim_{i\to\infty} \frac{b_1(G_i;\IF_p)}{[G:G_i]}\geq b_1^{(2)}(G)$
under the assumptions that $\cap G_i=\{1\}$ and each $G/G_i$ is a finite
$p$-group.
\end{abstract}

\maketitle


 \typeout{---------------------   Introduction ----------------------}

\setcounter{section}{0}
\section{Introduction}
\label{sec:Introduction}

\subsection{$\IQ$-approximation for the first $L^2$-Betti number}
Let $G$ be a finitely generated group. Given a field $K$, we let
$b_1(G;K) = \dim_{K}(H_1(G;K))$ be the first Betti number of $G$ with coefficients
in $K$ and $b_1(G)=b_1(G;\IQ)$ where $\IQ$ denotes the field of rational numbers.
Denote by $b_1^{(2)}(G)$ the \emph{first $L^2$-Betti number of $G$.}
Assuming that $G$ is finitely presented and residually finite, by L\"uck
Approximation Theorem (see~\cite{Lueck(1994c)}),
$b_1^{(2)}(G)$ can be approximated by normalized rational first Betti numbers of finite index subgroups of $G$:

\begin{theorem}[L\"uck approximation theorem]
\label{thm:Lueckapproximation}
Let $G$ be a finitely
presented residually finite group and $G=G_0\supseteq G_1\supseteq \ldots$
a descending chain of finite index normal subgroups of $G$, with
$\cap_{i\in\mathbb N}G_i=\{1\}$. Then
\begin{equation}
\label{Lueck_approx_equality}
b_1^{(2)}(G) = \lim_{i\to \infty}\frac{b_1(G_i)}{[G:G_i]}.
\end{equation}
\end{theorem}
\noindent
In the sequel we will occasionally refer to a descending chain $(G_i)$
of finite index normal subgroups of $G$ as a \emph{finite index normal chain} in $G$
and to the associated sequence
$\left(\frac{b_1(G_i)}{[G:G_i]}\right)_i$ as \emph{$\IQ$-approximation}.

  If we drop the assumption that $G$ is finitely presented, but still
  require that $\cap_{i\in\IN}G_i=\{1\}$, one still has inequality
  $b_1^{(2)}(G)\geq \limsup_{i\to \infty}\frac{b_1(G_i)}{[G:G_i]}$
by \cite[Theorem~1.1]{Lueck-Osin(2011)}, but equality need  not hold
\cite[Theorem~1.2]{Lueck-Osin(2011)}. The latter is proved in \cite{Lueck-Osin(2011)}
by constructing an example where $b_1^{(2)}(G)>0$, but
$\limsup_{i\to \infty}\frac{b_1(G_i)}{[G:G_i]}=0$ for any chain $(G_i)$
as above. In Section~\ref{sec:Q-approximation_without_limit} we will describe a variation of this construction
showing that the $\IQ$-approximation
$\left(\frac{b_1(G_i)}{[G:G_i]}\right)_i$ may not even have a limit:

\begin{theorem}
\label{Qappr_nolimit} There exists a finitely generated residually finite group $G$
and a descending chain $(G_i)_{i\in\IN}$ of finite index normal subgroups of $G$,
with $\cap_{i\in\IN}G_i=\{1\}$, such that $\lim_{i\to \infty}\frac{b_1(G_i)}{[G:G_i]}$ does not exist.
\end{theorem}

Another sequence we shall be interested in  is {\it $\IF_p$-approximation}, that is,$\left(\frac{b_1(G_i;\IF_p)}{[G:G_i]}\right)_i$, where
$\IF_p$ is the finite field of prime order $p$.
This sequence is particularly important under the additional assumption
that $(G_i)$ is a {\it $p$-chain}, that is, each $G_i$ has $p$-power index
(equivalently, $G/G_i$ is a finite $p$-group). In this case, $\left(\frac{b_1(G_i;\IF_p)}{[G:G_i]}\right)_i$
is monotone decreasing and therefore has a limit, often called \emph{$p$-gradient} or \emph{mod $p$ homology gradient} (see, e.g.,
\cite{Lackenby(2007covering)}).

Since obviously $b_1(H)\leq b_1(H;\IF_p)$ for any group $H$,
one always has inequality
\begin{equation}
\label{inequality_trivial}
\limsup_{i\to \infty}\frac{b_1(G_i)}{[G:G_i]}\leq
\limsup_{i\to \infty}\frac{b_1(G_i;\IF_p)}{[G:G_i]},
\end{equation}
and it is natural to ask for sufficient conditions under which
equality holds. Of particular interest is the case when
$G$ is finitely presented and $\cap_{i\in\IN}G_i=\{1\}$
when $\IQ$-approximation does have a limit by
Theorem~\ref{thm:Lueckapproximation}.

  \begin{question}[$\IQ$-approximation and $\IF_p$-approximation]
    \label{que:Q-approximation_versus-F_p_approximation}
For which finitely presented groups $G$ and finite index normal chains $(G_i)$ with $\cap_{i\in\IN}G_i=\{1\}$
do we have equality
    \[
    \lim_{i\to \infty}\frac{b_1(G_i)}{[G:G_i]}=
\lim_{i\to \infty}\frac{b_1(G_i;\IF_p)}{[G:G_i]}
{?}
    \]
  \end{question}

If $G$ is not finitely presented, the above equality need not hold
even if we require that $(G_i)$ is a $p$-chain. Indeed,
as proved in \cite{Osin(2011_rankgradient)} and independently in
\cite{Schlage-Puchta(2012)}, there exists a $p$-torsion
residually-$p$ group $G$ with $\lim_{i\to \infty}\frac{b_1(G_i;\IF_p)}{[G:G_i]}>0$
for any $p$-chain $(G_i)$ in $G$ (and since $G$ is residually-$p$,
we can choose a $p$-chain with $\cap G_i=\{1\}$). Since $b_1(H)=0$ for any torsion group $H$, we have $\lim_{i\to \infty}\frac{b_1(G_i)}{[G:G_i]}=0$
for such group $G$.

In Section~\ref{sec:A_counterexample_with_non-trivial_intersection} we give
an example showing that the answer to Question~\ref{que:Q-approximation_versus-F_p_approximation} would 
also become negative if we drop the assumption $\cap_{i\in\IN}G_i=\{1\}$, even if $G$ is finitely presented and $(G_i)$ is a $p$-chain which has infinitely many distinct terms.

\subsection{Comparing $\IF_p$-approximation and first $L^2$-Betti number}

Since both $\IF_p$-approximation and the first $L^2$-Betti number
provide upper bounds for $\IQ$-approximation, it is natural to ask how
the former two quantities are related to each other. We address this question
in the case of $p$-chains.
  \begin{theorem}
    \label{cor:The_first_L2-Betti_number_and_F_p-approximation}
 Let $p$ be a prime number. Let $G$ be a finitely generated group
and  $G = G_0 \supseteq G_1 \supseteq G_2 \supseteq \cdots$ a descending
chain of normal subgroups of $G$ of $p$-power index.
Then
\begin{enumerate}
\item \label{cor:The_first_L2-Betti_number_and_F_p-approximation:monotone}
The sequence $\left(\frac{b_1(G_i;\IF_p)}{[G:G_i]}\right)_i$ is monotone decreasing and therefore
converges;
\item \label{cor:The_first_L2-Betti_number_and_F_p-approximation:inequality}
Assume that $\bigcap_{i\in \IN} G_i=\{1\}$.
Then
\[
b_1^{(2)}\bigl(G) \le \lim_{i \to \infty} \frac{b_1(G_i;\IF_p)}{[G:G_i]}.
\]
\end{enumerate}
\end{theorem}
\noindent
We note that for finitely presented groups Theorem~\ref
{cor:The_first_L2-Betti_number_and_F_p-approximation}(2) is a straightforward consequence
of Theorem~\ref{thm:Lueckapproximation}.

We provide two different proofs of Theorem~\ref
{cor:The_first_L2-Betti_number_and_F_p-approximation}.
First, Theorem~\ref{cor:The_first_L2-Betti_number_and_F_p-approximation}
is a special case of Theorem~\ref{the:The_first_L2-Betti_number_and_F_p-approximation},  which will be proved in Section~\ref{sec:The_first_L2-Betti_number_and_approximation_in_prime_characteristic}.
  An alternative proof of Theorem~\ref{cor:The_first_L2-Betti_number_and_F_p-approximation} given
in Section~\ref{sec:Alternative_proof_of_main_Corollary} will be based
on Theorem~\ref{thm:Lueck_notes}. The latter may be of independent interest
and has another important corollary, which can be considered as an extension of Theorem~\ref{thm:Lueckapproximation} to groups which are finitely presented, but not necessarily residually finite.
Here is a slightly simplified version of Theorem~\ref{thm:Lueck_notes}.
\begin{theorem}
\label{thm:Lueck_notes_simplified}
Let $G$ be a finitely presented group, and let $K$ be the kernel
of the canonical map from $G$ to its profinite completion or
pro-$p$ completion for some prime $p$. Let
$(G_i)$ be a descending chain of finite index normal subgroups of $G$ such
that $\cap_{i\in\IN} G_i=K$ (note that such a chain always exists).
 Then
$$b_1^{(2)}(G/K)=\lim_{i\to\infty} \frac{b_1(G_i)}{[G:G_i]}.
$$
\end{theorem}

\subsection{Connection with rank gradient}

Let $G$ be a finitely generated group. In the sequel we denote by $d(G)$ the
minimal number of generators, sometimes also called \emph{the rank of $G$}.
Let $(G_i)_{i\in\IN}$ be a descending chain of finite index normal
subgroups of $G$. The \emph{rank gradient of $G$} (with
respect to $(G_i)$), denoted by $\RG(G;(G_i))$, is defined by
\begin{eqnarray}
\RG(G;(G_i)) &= & \lim_{i \to\infty} \frac{d(G_i) -1}{[G:G_i]}.
\label{rank_gradient}
\end{eqnarray}
The above limit always exists since for any finite index subgroup $H$ of $G$ one has $\frac{d(H)-1}{[G:H]} \le d(G)-1$ by the Schreier index formula.

Rank gradient was originally introduced by Lackenby~\cite{Lackenby(2005expanders)} as a tool for studying 3-manifold groups, but is also interesting from a purely group-theoretic point of view
(see, e.g., \cite{Abert-Jaikin-Zapirain-Nikolov(2011),Abert-Nikolov(2012),Osin(2011_rankgradient),Schlage-Puchta(2012)}).

Provided that $G$ is infinite and
$\bigcap_{i\in \IN} G_i=\{1\}$, the following inequalities
are known to hold:
\begin{equation}
\label{inequalities1}
\RG(G;(G_i))\geq {\rm cost}(G)-1\geq b_1^{(2)}(G).
\end{equation}
The first inequality was proved by Ab\'ert and Nikolov~\cite[Theorem~1]{Abert-Nikolov(2012)}, and the second one is due
to Gaboriau~\cite[Corollaire~3.16,~3.23]{Gaboriau(2002a)} (see \cite{Gaboriau(2000b),Gaboriau(2002a),Gaboriau(2002b)}
for the definition and some key results about cost).

It is not known if either inequality in \eqref{inequalities1}
can be strict. In particular, the following question is open.

\begin{question}
\label{que_RG_L2Betti}
Let $G$ be an infinite finitely generated residually finite group
and $(G_i)$ a descending chain of finite index  normal subgroups of $G$
with $\cap_{i\in \IN} G_i=\{1\}$. Is it always true that
$$\RG(G;(G_i))= b_1^{(2)}(G)?$$
\end{question}
\noindent
Theorem~\ref{cor:The_first_L2-Betti_number_and_F_p-approximation}
provides a potentially new approach for answering Question~\ref{que_RG_L2Betti}
in the negative, as explained below.

In view of the obvious inequality $d(H)\geq b_1(H;K)$ for any group $H$
and any field $K$, one always has
$\RG(G;(G_i)) \geq \limsup_{i \to\infty} \frac{b_1(G_i;K)}{[G:G_i]}$.

\begin{question}\label{que:RG_and_limits_of_Betti_numbers}
For which infinite finitely generated groups $G$, finite index normal chains $(G_i)_{i \in \IN}$ 
with $\bigcap_{i \in \IN} G_i = \{1\}$ and fields $K$, do  we have
\begin{equation}
\label{RG_Betti_equality}
\RG(G;(G_i)) =  \limsup_{i \to\infty} \frac{b_1(G_i;K)}{[G:G_i]}?
\end{equation}
\end{question}
\begin{remark} Since for a group $H$, the first Betti number $b_1(H;K)$
depends only on the characteristic of $K$, one can assume
that $K=\IQ$ or $K=\IF_p$ for some $p$. The same remark applies to
Question~\ref{que:limit_independent_of_chain} below.
\end{remark}

Note that if $K=\IQ$, equality \eqref{RG_Betti_equality} does not hold in general --
if it did, Theorem~\ref{Qappr_nolimit} would have implied the existence of a group $G$
and a finite index normal chain $(G_i)$ in $G$ for which the sequence $\left(\frac{d(G_i)-1}{[G:G_i]}\right)_i$
has no limit, which is impossible since this sequence is monotone decreasing. 
If one can find a group $G$ for which \eqref{RG_Betti_equality}
fails with $K=\IF_p$ and $(G_i)$ a $p$-chain, then in view of
Theorem~\ref{cor:The_first_L2-Betti_number_and_F_p-approximation} such group $G$ would answer
Question~\ref{que_RG_L2Betti} in the negative.

The answer to Question~\ref{que:RG_and_limits_of_Betti_numbers}
would become negative if we drop the assumption $\cap G_i=\{1\}$
even if $G$ is finitely presented and $(G_i)$ is a $p$-chain (with infinitely
many distinct terms), as we will see in Section~\ref{sec:A_counterexample_with_non-trivial_intersection}.

\subsection{Independence of the chain}

So far we discussed the dependence of the quantity $\limsup_{i \to\infty} \frac{b_1(G_i;K)}{[G:G_i]}$
on the field $K$, but perhaps an even more important question is when it is independent of the chain.
Again it is reasonable to require that $\bigcap_{i \in \IN} G_i = \{1\}$ since without this
restriction the answer would be negative already for very nice groups like $F\times \IZ$,
where $F$ is a non-abelian free group. Note that independence of $\limsup_{i \to\infty} \frac{b_1(G_i;K)}{[G:G_i]}$ of the chain $(G_i)$ as above automatically implies that $\lim_{i \to\infty} \frac{b_1(G_i;K)}{[G:G_i]}$ must exist.

\begin{question}\label{que:limit_independent_of_chain}
For which finitely generated residually finite groups $G$ and fields $K$  does the limit $\lim_{i \to\infty} \frac{b_1(G_i;K)}{[G:G_i]}$
exist for all finite index normal chains $(G_i)_{i\in \IN}$ with $\bigcap_{i \in \IN} G_i = \{1\}$ and is independent of the choice of the chain $(G_i)$?
\end{question}

The answer to Question~\ref{que:limit_independent_of_chain} is known to be positive
if $K=\IQ$ and either $G$ is finitely presented (by Theorem~\ref{thm:Lueckapproximation})
or $G$ is a limit of left orderable amenable groups in the space of marked
group presentations, in which case equality \eqref{Lueck_approx_equality} holds
by \cite[Corollary~1.5]{Osin-Thom(2013)}. Question~\ref{que:limit_independent_of_chain}
remains open if $G$ is finitely presented and $K=\IF_p$. If $G$ is arbitrary, the answer may be
negative  for any $K$ -- this follows directly from Theorem~\ref{Qappr_nolimit} if $K=\IQ$ and
from its stronger version Theorem~\ref{thm:Qapproximation} if $K=\IF_p$.
In the latter case, however, it is natural to impose the additional
assumption that $(G_i)$ is a $p$-chain, which does not hold in our examples.

Essentially the only case when answer to Question~\ref{que:limit_independent_of_chain} is known
to be positive for all fields is when $G$ contains a normal infinite amenable
subgroup (e.g., if $G$ itself is infinite amenable). In this case, $\RG(G;(G_i)) = 0$ for all finite index normal chains $(G_i)$ with trivial intersection, as proved by Lackenby~\cite[Theorem~1.2]{Lackenby(2005expanders)}
when $G$ is finitely presented and by Ab\'ert and Nikolov~\cite[Theorem~3]{Abert-Nikolov(2012)}
in general. This, of course, implies that in such groups
$\lim_{i \to\infty} \frac{b_1(G_i;K)}{[G:G_i]}=0$
for any such chain $(G_i)$ and hence the answer to
Questions~\ref{que:RG_and_limits_of_Betti_numbers}
and~\ref{que:limit_independent_of_chain} is positive.

Finally, we comment on the status of a more general version of Question~\ref{que:limit_independent_of_chain}:

\begin{question}\label{que:approximation_in_prime_characteristic}
  For which residually finite groups $G$, fields $K$, finite index normal chains $(G_i)$ with
  $\bigcap_{i  \in \IN} G_i = \{1\}$, free $G$-$CW$-complexes $X$ of finite type and natural
  numbers $n$, does the limit
  $\lim_{i \to \infty} \frac{b_n(G_i\backslash X;K))}{[G:G_i]}$ exist and is independent of the chain?
\end{question}

Again, if $K$ has characteristic zero, the answer is always yes and the limit can be identified with
the $n$-th $L^2$-Betti number  $b_n^{(2)}(X;\caln(G))$
 (see~\cite{Lueck(1994c)} or~\cite[Theorem~13.3~(2) on page~454]{Lueck(2002)},
which is a generalization of Theorem~\ref{thm:Lueckapproximation}).
If $K$ has positive characteristic, the answer is yes if $G$ is virtually torsion-free elementary amenable, in which case the limit can be identified 
with the Ore dimension of $H_n(X;K)$ (see~\cite[Theorem~5.3]{Linnell-Lueck-Sauer(2011)});
the answer is also yes for any finitely generated amenable group $G$ -- this follows from 
\cite[Theorem~17]{Abert-Jaikin-Zapirain-Nikolov(2011)} or \cite[Theorem~2.1]{Linnell-Lueck-Sauer(2011)} --
and the limit can be described using Elek dimension function (see \cite{Elek(2003c)}). There are examples
for $G = \IZ$ of finite $G$-$CW$-complexes $X$ where the limits 
$\lim_{i \to \infty} \frac{b_n(G_i\backslash X;K))}{[G:G_i]}$ are different for $K = \IQ$ and $K = \IF_p$ (but $X$ is not $EG$), see~\cite[Example~6.2]{Linnell-Lueck-Sauer(2011)}.

\subsection{Acknowledgments}
The authors want to thank the American Institute of Mathematics for its hospitality
during their stay at the Workshop ``$L^2$-invariants and their relatives for finitely generated groups''
organized by Mikl\'os Ab\'ert, Mark Sapir, and Dimitri Shlyakhtenko
in September 2011, where some of the ideas of this paper were developed.
The authors are very grateful to Denis Osin for proposing several improvements
in Section~\ref{sec:A_counterexample_with_non-trivial_intersection} and other
useful discussions. The first author is very grateful to Andrei Jaikin-Zapirain for many helpful discussions
related to the subject of this paper, sending his unpublished work
``On $p$-gradient of finitely presented groups'' and suggesting a stronger version
of Theorem~\ref{thm:Lueck_notes}(2).


 \typeout{--   Section 2: The first $L^2$-Betti number and approximation in prime characteristic ---------------}

\section{The first $L^2$-Betti number and approximation in prime characteristic}
\label{sec:The_first_L2-Betti_number_and_approximation_in_prime_characteristic}

If $G$ is a group  and $X$ a $G$-$CW$-complex, we denote by
\begin{eqnarray}
b_n^{(2)}(X; \caln(G))
& = &
\dim_{\caln(G)}\bigl(H_n(\caln(G) \otimes_{\IZ G} C_*(X))\bigr)
\label{n-th_L2-Betti_number}
\end{eqnarray}
its \emph{$n$-th $L^2$-Betti number}. Here $C_*(X)$ is the cellular $\IZ G$-chain complex of $X$,
$\caln(G)$ is the group von Neumann algebra and $\dim_{\caln(G)}$ is the dimension function
for (algebraic) $\caln(G)$-modules in the sense
of~\cite[Theorem~6.7 on page~239]{Lueck(2002)}. Notice that
$b_1^{(2)}(G) = b_1^{(2)}(EG;\caln(G))$.

The goal of this section is to prove the following theorem which generalizes Theorem~\ref{cor:The_first_L2-Betti_number_and_F_p-approximation}:

\begin{theorem}[The first $L^2$-Betti number and $\IF_p$-approximation]
  \label{the:The_first_L2-Betti_number_and_F_p-approximation}
  Let $p$ be a prime number. Let $G$ be a finitely generated group
and $(G_i)$ a descending chain of normal subgroups of $p$-power index in $G$.
Let $K= \bigcap_{i\in \IN} G_i$.
Then the sequence $\left(\frac{b_1(G_i;\IF_p)}{[G:G_i]}\right)_i$ is monotone
decreasing, the limit $\lim_{i \to \infty} \frac{b_1(G_i;\IF_p)}{[G:G_i]}$
exists and satisfies
  \[
  b_1^{(2)}\bigl(K\backslash EG;\caln(G/K)\bigr) \le \lim_{i \to \infty}
  \frac{b_1(G_i;\IF_p)}{[G:G_i]}.
  \]
\end{theorem}

For its proof we will need the following lemma, which is proved 
in~\cite[Lemma~4.1]{Bergeron-Linnell-Lueck-Sauer(2012)}, although it was probably
well known before.

\begin{lemma} \label{lem:monotonicity}
Let $p$ be a prime and $m,n$ positive integers. Let $H$ be a finite $p$-group.
Consider an $\IF_pH$-map $\alpha \colon \IF_pH^m \to \IF_pH^n$.
Define the $\IF_p$-map
\[
\overline{\alpha}
= \id_{\IF_p} \otimes_{\IF_pH} \alpha \colon \IF_p^m
= \IF_p\otimes_{\IF_pH} \IF_pH^m  \to   \IF_p^n = \IF_p\otimes_{\IF_pH}  \IF_pH^n,
\]
where we consider $\IF_p$ as $\IF_p H$-module by the trivial $H$-action. Then
\[
\dim_{\IF_p}(\im(\alpha)) \ge |H| \cdot \dim_{\IF_p}(\im(\overline{\alpha})).
\]
\end{lemma}

Notice that the assertion of Lemma~\ref{lem:monotonicity} is not true
if we do not require that $H$ is a $p$-group or if we replace $\IF_p$ by
a field of characteristic not equal to $p$.

\begin{proof}[Proof of
  Theorem~\ref{the:The_first_L2-Betti_number_and_F_p-approximation}]
  Since $G$ is finitely generated, there is a $CW$-model for $BG$ with one
  $0$-cell and a finite number, let us say $s$, of $1$-cells.  Let $EG \to BG$
  be the universal covering. Put $X = K\backslash EG$ and $Q= G/K$. Then $X$ is a free
  $Q$-$CW$-complex with finite $1$-skeleton. Its cellular $\IZ Q$-chain complex
  $C_*(X)$ looks like
  \[
  \cdots \to C_2(X) = \bigoplus_{j=1}^{r} \IZ Q\xrightarrow{c_2} C_1(X) =
  \bigoplus_{j=1}^s \IZ Q \xrightarrow{c_1}  C_0(X) = \IZ Q
  \]
  where $r$ is a finite number or infinity.

For $m = 0,1,2, \ldots$ we define a
  $\IZ Q$-submodule of $C_2(X)$ by $C_2(X)|_m = \bigoplus_{j= 1}^{\max\{m,r\}} \IZ Q$.
  Denote by $c_2|_m \colon C_2(X)|_m \to C_1(X)$ the restriction of
  $c_2$ to $C_2(X)|_m$.

  Consider a $\IZ Q$-map $f \colon M\to N$. Denote by
  $f^{(2)} \colon M^{(2)} \to N^{(2)}$ the $\caln(Q)$-homomorphism
  $\id_{\caln(G)}   \otimes_{\IZ Q} f \colon \caln(Q) \otimes_{\IZ Q} M \to \caln(Q) \otimes_{\IZ Q}  N$.
  Put $Q_i = G_i/K$. Let $f[i] \colon M[i] \to N[i]$ be the $\IQ$-homomorphism
   $\id_{\IQ}   \otimes f \colon \IQ \otimes_{\IZ[Q_i]} M \to \IQ \otimes_{\IZ[Q_i]} N$. Denote
  by $f[i,p] \colon M[i,p] \to N[i,p]$  the $\IF_p$-homomorphism
  $\id_{\IF_p}   \otimes_{\IZ[Q_i]} f \colon \IF_p \otimes_{\IZ[Q_i]} M \to \IF_p \otimes_{\IZ[Q_i]} N$.
  If $M = \bigoplus_{j=1}^t \IZ Q$, then $M^{(2)} = \bigoplus_{j=1}^t \caln(Q)$,
  $M[i] = \bigoplus_{j=1}^t \IZ [Q/Q_i]$ and $M[i,p] = \bigoplus_{j=1}^t \IF_p[Q/Q_i]$.
  
  Note that
\begin{equation*}
\label{Betti_equality} 
b_1(Q_i\backslash X;\IF_p)=b_1(G_i\backslash EG;\IF_p)=b_1(BG_i;\IF_p)=
b_1(G_i;\IF_p).
\end{equation*}
  Since all dimension functions are additive (see~\cite[Theorem~6.7 on page~239]{Lueck(2002)}),
we conclude
  \begin{eqnarray}
  b_1^{(2)}\bigl(X;\caln(Q)\bigr) & = & s - 1 - \dim_{\caln(Q)}\bigl(\im(c_2^{(2)})\bigr);
  \label{b_1(2)_and_im}
  \\
  \frac{b_1\bigl(G_i;\IF_p)}{[Q:Q_i]}
  & = &
  s - 1 - \frac{\dim_{\IF_p}\bigl(\im(c_2[i,p])\bigr)}{[Q:Q_i]};
  \label{b_1(F_p)_and_im}
  \\
   \dim_{\caln(Q)}\bigl(\im(c_2|_m^{(2)})\bigr)
   & = &
   m - \dim_{\caln(Q)}\bigl(\ker(c_2|_m^{(2)})\bigr);
  \label{im_ker_(2)}
  \\
  \frac{\dim_{\IQ}\bigl(\im(c_2|_m [i])\bigr)}{[Q:Q_i]}
   & = &
   m - \frac{\dim_{\IQ}\bigl(\ker(c_2|_m [i])\bigr)}{[Q:Q_i]};
  \label{im_and_ker_i}
  \\
   \frac{\dim_{\IF_p}\bigl(\im(c_2|_m[i,p])\bigr)}{[Q:Q_i]} & = & m - \frac{\dim_{\IF_p}\bigl(\ker(c_2|_m[i,p])\bigr)}{[Q:Q_i]}.
  \label{im_and_ker_i_p}
 \end{eqnarray}
There is an  isomorphism of $\IF_p$-chain complexes
$\IF_p \otimes_{\IF_p[Q_{i+1}\backslash Q_i]}  C_*(X)[(i+1),p]  \xrightarrow{\cong} C_*(X)[i,p]$,
where the $Q_{i+1}\backslash Q_i$-operation on $C_*(X)[i+1]$ comes from the identification
$C_*(X)[i+1] = \IF_p \otimes_{\IF_p[Q_{i+1}]} C_*(X) = \IF_p[Q_{i+1}\backslash Q] \otimes_{\IF_p Q} C_*(X)$.
This is compatible with the passage from $C_2(X)$ to $C_2(X)|_m$. Hence $c_2|_m[i,p]$ can be identified with
$\id_{\IF_p}  \otimes_{\IF_p[Q_{i+1}\backslash Q_i]} c_2|_m[(i+1),p]$. Since $Q_{i+1}\backslash Q_i$ is a finite $p$-group,
Lemma~\ref{lem:monotonicity} implies
\begin{eqnarray*}
\dim_{\IF_p}\bigl(\im(c_2|_m[(i+1),p])\bigr)
& \ge &  [Q_i:Q_{i+1}] \cdot \dim_{\IF_p}\bigl(\im(c_2|_m[i,p])\bigr).
\end{eqnarray*}
We conclude
\begin{eqnarray}
\frac{\dim_{\IF_p}\bigl(\im(c_2|_m[(i+1),p])\bigr)}{[Q:Q_{i+1}]}
& \ge &
\frac{\dim_{\IF_p}\bigl(\im(c_2|_m[i,p])\bigr)}{[Q:Q_i]}.
\label{monotonicity}
\end{eqnarray}
Since $\im(c_2^{(2)}) = \bigcup_{m} \im(c_2|_m^{(2)})$ and
$\im(c_2[i,p]) = \bigcup_{m} \im(c_2|_m[i,p])$ and the dimension functions are compatible with directed unions
(see~\cite[Theorem~6.7 on page~239]{Lueck(2002)}), we get
\begin{eqnarray}
\dim_{\caln(Q)}\bigl(\im(c_2^{(2)})\bigr) & = &
\lim_{m \to \infty} \dim_{\caln(Q)}\bigl(\im(c_2|_m^{(2)})\bigr);
\label{cofinal_(2)}
\\
\dim_{\IF_p}\bigl(\im(c_2[i,p])\bigr) & = &
\lim_{m \to \infty}  \dim_{\IF_p}\bigl(\im(c_2|_m[i,p])\bigr).
\label{cofinal_F_p}
\end{eqnarray}
We conclude from~\cite[Theorem~13.3~(2) on page~454 and Lemma~13.4 on page~455]{Lueck(2002)}
\begin{eqnarray*}
\lim_{i \to \infty} \frac{\dim_{\IQ}\bigl(\ker(c_2|_m[i])\bigr) }{[Q:Q_i]}
& = &
\dim_{\caln(Q)}\bigl(\ker(c_2|_m^{(2)})\bigr).
\end{eqnarray*}
This implies together with~\eqref{im_ker_(2)} and~\eqref{im_and_ker_i}
\begin{eqnarray}
\lim_{i \to \infty} \frac{\dim_{\IQ}\bigl(\im(c_2|_m[i])\bigr) }{[Q:Q_i]}
& = &
\dim_{\caln(Q)}\bigl(\im(c_2|_m^{(2)})\bigr).
\label{classical_approximation}
\end{eqnarray}
Finally, it is easy to see that
\begin{eqnarray}
\dim_{\IQ}\bigl(\im(c_2|_m[i])\bigr)
& \ge &
\dim_{\IF_p}\bigl(\im(c_2|_m[i,p])\bigr).
\label{Q_ge_F_p}
\end{eqnarray}
\medskip
Putting everything together, we can now prove both assertions of Theorem~\ref{the:The_first_L2-Betti_number_and_F_p-approximation}. First, for a fixed $m$, the sequence $\left(\frac{\dim_{\IF_p}\bigl(\im(c_2|_m[i,p])\bigr)}{[Q:Q_i]}\right)_i$ is monotone increasing
by \eqref{monotonicity}, whence the sequence 
$\left(\frac{\dim_{\IF_p}\bigl(\im(c_2[i,p])\bigr)}{[Q:Q_i]}\right)_i$ is also monotone increasing
by \eqref{cofinal_F_p} and therefore the sequence
$\left(\frac{b_1(G_i;\IF_p)}{[Q:Q_i]}\right)_i$ is monotone decreasing by \eqref{b_1(F_p)_and_im}.
This proves the first assertion of Theorem~\ref{the:The_first_L2-Betti_number_and_F_p-approximation}
since clearly $[Q:Q_i]=[G:G_i]$.

Inequality  \eqref{monotonicity} also implies that $\lim\limits_{i \to \infty} \frac{\dim_{\IF_p}\bigl(\im(c_2|_m[i,p])\bigr)}{[Q:Q_i]}\geq 
\frac{\dim_{\IF_p}\bigl(\im(c_2|_m[j,p])\bigr)}{[Q:Q_j]}$ for any fixed $j$ and $m$, and so
\begin{equation}
\label{change_of_order}
\lim_{m\to\infty}\lim_{i \to \infty} \frac{\dim_{\IF_p}\bigl(\im(c_2|_m[i,p])\bigr)}{[Q:Q_i]}\geq 
\sup_{i\geq 0}\left\{\lim_{m \to \infty}  \frac{\dim_{\IF_p}\bigl(\im(c_2|_m[i,p])\bigr)}{[Q:Q_i]} \; \right\}.
\end{equation}
Therefore,
\begin{eqnarray*}
b_1^{(2)}(X;\caln(Q))
& \stackrel{\eqref{b_1(2)_and_im}}{=} &
s - 1 - \dim_{\caln(Q)}\bigl(\im(c_2^{(2)})\bigr)
\\
& \stackrel{\eqref{cofinal_(2)}}{=} &
s - 1 - \lim_{m \to \infty} \dim_{\caln(Q)}\bigl(\im(c_2|_m^{(2)})\bigr)
\\
& \stackrel{\eqref{classical_approximation}}{=} &
s - 1 - \lim_{m \to \infty} \lim_{i \to \infty} \frac{\dim_{\IQ}\bigl(\im(c_2|_m[i])\bigr)}{[Q:Q_i]}
\\
& \stackrel{\eqref{Q_ge_F_p}}{\le}  &
s - 1 - \lim_{m \to \infty} \lim_{i \to \infty} \frac{\dim_{\IF_p}\bigl(\im(c_2|_m[i,p])\bigr)}{[Q:Q_i]}
\\
& \stackrel{\eqref{change_of_order}}{\le} &
s - 1 - \sup_{i\geq 0}\left\{\lim_{m \to \infty}  \frac{\dim_{\IF_p}\bigl(\im(c_2|_m[i,p])\bigr)}{[Q:Q_i]}\right\}
\\
& \stackrel{\eqref{cofinal_F_p}}{=} &
s - 1 - \sup_{i\geq 0}\left\{\frac{\dim_{\IF_p}\bigl(\im(c_2[i,p])\bigr)}{[Q:Q_i]}\right\}
\\
& =  &
\inf_{i\geq 0}\left\{s -1- \frac{\dim_{\IF_p}\bigl(\im(c_2[i,p])\bigr)}{[Q:Q_i]}\right\}
\\
& \stackrel{\eqref{b_1(F_p)_and_im}}{=} &
\inf_{i\geq 0}\left\{\frac{b_1(G_i;\IF_p)}{[Q:Q_i]}\right\}.
\end{eqnarray*}
This finishes the proof of Theorem~\ref{the:The_first_L2-Betti_number_and_F_p-approximation}.
\end{proof}


 \typeout{-----------------   Section 3: Alternative proof of main Corollary ----------------------------}

\section{Alternative proof of Theorem~\ref{cor:The_first_L2-Betti_number_and_F_p-approximation}}
\label{sec:Alternative_proof_of_main_Corollary}

In this section we give an alternative proof of Theorem~\ref{cor:The_first_L2-Betti_number_and_F_p-approximation}.
Namely, Theorem~\ref{cor:The_first_L2-Betti_number_and_F_p-approximation}
is an easy consequence of the following result, which may be useful in its own right.

\begin{theorem}
\label{thm:Lueck_notes}
Let $G$ be a finitely presented group, let $(G_i)$ be a descending
chain of finite index normal subgroups of $G$, and let $K=\bigcap_{i=1}^{\infty} G_i$.
\begin{enumerate}
\item \label{thm:Lueck_notes:inequalities}
The following inequalities hold:
\begin{multline*}
\qquad \lim_{i\to\infty} \frac{b_1(G_i/K)}{[G:G_i]}
\leq b_1^{(2)}(G/K)
\leq b_1^{(2)}\bigl(K\backslash EG; \caln(G/K)\bigr)=\lim_{n\to\infty} \frac{b_1(G_i)}{[G:G_i]}.
\end{multline*}

\item \label{thm:Lueck_notes:equalities}
 Let $\mathcal{C}$ be any class of finite groups which is closed under subgroups,
extensions (and isomorphisms) and contains at least one non-trivial group (for instance, $\mathcal C$ could
be the class of all finite groups or all finite $p$-groups for a fixed prime $p$).
Assume that $K$ is the kernel of the canonical map from $G$ to its pro-$\mathcal{C}$ completion.
Then
$$b_1^{(2)}(G/K)=\lim_{i\to\infty} \frac{b_1(G_i)}{[G:G_i]}.$$
If in addition all groups $G/G_i$ are in $\mathcal C$, then
\begin{multline}
\label{equality_approximation}
\qquad \lim_{i\to\infty} \frac{b_1(G_i/K)}{[G:G_i]}
=  b_1^{(2)}(G/K)
= b_1^{(2)}\bigl(K\backslash EG; \caln(G/K)\bigr)=\lim_{i\to\infty} \frac{b_1(G_i)}{[G:G_i]}.
\end{multline}
\end{enumerate}
\end{theorem}
\begin{proof}~\eqref{thm:Lueck_notes:inequalities}
  Since $G$ is finitely presented, there is a $G$-$CW$-model for the classifying
  space $BG$ whose $2$-skeleton is finite. Let $EG \to BG$ be the universal
  covering. Then $EG$ is a free $G$-$CW$-complex with finite $2$-skeleton. Put
  \begin{eqnarray*}
   Q = G/K;
   \\
   Q_i = G_i/K.
 \end{eqnarray*}
Then $Q = Q_0 \supseteq Q_1 \supseteq \cdots$
is a descending chain of finite index normal subgroups of $Q$ with  
$\bigcap_{i = 0}^{\infty} Q_i =  \{1\}$ and we have for $i = 0,1,2,\ldots$
  \begin{eqnarray}
  [G:G_i]  & = & [Q:Q_i].
    \label{equality_of_degrees}
  \end{eqnarray}
The quotient $X= K\backslash EG$ is a free $Q$-$CW$-complex whose
$2$-skeleton is finite.    
Let $X_2$ be the $2$-skeleton of $X$. Since the first $L^2$-Betti number and 
the first Betti number depend only on the $2$-skeleton,
from~\cite[Theorem~0.1]{Lueck(1994c)} applied to the $G$-covering $X_2\to  X_2/G$ (we do not need $X_2$ to be simply connected) or directly from~\cite[Theorem~13.3 on page~454]{Lueck(2002)}, we obtain
 \begin{eqnarray}
 b_1^{(2)}(X;\caln(Q)) & = & \lim_{i \to \infty} \frac{b_1(Q_i\backslash X)}{[Q:Q_i]}.
\label{class_approx}
\end{eqnarray}
Let $f \colon X \to EQ$ be the classifying map. Since $EQ$ is simply connected,
this map is $1$-connected. This implies by~\cite[Theorem~6.54~(1a) on page~265]{Lueck(2002)}
\begin{eqnarray}
b_1^{(2)}(X;\caln(Q)) & \ge & b_1^{(2)}(EQ;\caln(Q)).
\label{inequality_of_L2-Betti_numbers}
\end{eqnarray}
The group $Q$ is finitely generated (but not necessarily finitely presented), so
by \cite[Theorem~1.1]{Lueck-Osin(2011)} we have
\begin{eqnarray}
\lim_{i \to \infty} \frac{b_1(Q_i)}{[Q:Q_i]} & \le & b_1^{(2)}(Q).
\label{Lueck-Osin}
\end{eqnarray}
Notice that $b_1^{(2)}(Q)=b_1^{(2)}(EQ;\caln(Q))$ by definition and we obviously have
$Q_i \backslash X = G_i \backslash EG = BG_i$ and hence $b_1(Q_i\backslash X) = b_1(G_i)$. Combining
\eqref{equality_of_degrees},~\eqref{class_approx},~\eqref{inequality_of_L2-Betti_numbers}, and~\eqref{Lueck-Osin}, we get
$$\lim_{i \to \infty} \frac{b_1(Q_i)}{[Q:Q_i]} \le
b_1^{(2)}(Q) \le
b_1^{(2)}(X;\caln(Q)) = 
\lim_{i \to \infty} \frac{b_1(Q_i\backslash X)}{[Q:Q_i]} = 
\lim_{i \to \infty} \frac{b_1(G_i)}{[G:G_i]}.$$
This finishes the proof of assertion~\eqref{thm:Lueck_notes:inequalities}.
\\[2mm]~\eqref{thm:Lueck_notes:equalities} First observe that
since $b_1^{(2)}\bigl(K\backslash EG; \caln(G/K)\bigr)=\lim\limits_{i\to\infty} \frac{b_1(G_i)}{[G:G_i]}$
by $\eqref{thm:Lueck_notes:inequalities}$, the limit $\lim\limits_{i\to\infty} \frac{b_1(G_i)}{[G:G_i]}$
is the same for all finite index normal chains $(G_i)$ with $\cap_{i\in\IN}G_i=K$.%
\footnote{We are grateful to Andrei Jaikin-Zapirain for this observation.}
By definition of $K$, there exists at least one such chain with $G/G_i\in\mathcal C$ for all $i$
(e.g., we can let $(G_i)$ be a base of neighborhoods of $1$ for the pro-$\mathcal C$ topology on $G$), 
so it suffices to prove \eqref{equality_approximation}. Thus, from now on we will assume that $G/G_i\in\mathcal C$ 
for $i\in\IN$.
\\[2mm]
For a finitely generated group $H$
we denote by $H'$ the kernel of the composite of canonical projections $H \to
H_1(H) \to H_1(H)/\tors(H_1(H))$, so that $H/H'$ is a free abelian group
of rank $b_1(H)$.

As in the proof of (1), we put $Q_i=G_i/K$ for $i\in\IN$. It is sufficient
to prove that that $K\subseteq G_i'$ for $i\in \IN$. Indeed, this would imply that
$Q_i/Q_i'\cong G_i/G_i'$, whence $b_1(Q_i)=b_1(G_i)$ and therefore
$\lim_{i\to\infty} \frac{b_1(Q_i)}{[G:G_i]}=\lim_{i\to\infty} \frac{b_1(G_i)}{[G:G_i]}$,
which proves (2) in view of \eqref{thm:Lueck_notes:inequalities}.
\\[2mm]
Fix $i\in\IN$ and let $H=G_i$. Since $\mathcal{C}$ contains at least one non-trivial finite group 
and is closed under subgroups, it contains a finite cyclic group, say of order $k$. 
Since $\mathcal{C}$ is closed under extensions, it contains
$\left(\IZ/k^m\IZ\right)^b$ for all $m,b\in \IN$. Setting $b=b_1(H)$, we get
that $H/H' H^{k^m}\in\mathcal C$ for all $m\in\IN$, and since $\mathcal C$
is closed under extensions, we obtain $G/H' H^{k^m}\in\mathcal C$.
By definition, $K$ is the intersection of all normal subgroups $L$ of $G$
with $G/L\in\mathcal C$.  Therefore, $K\subseteq \bigcap\limits_{m\in\IN} H' H^{k^m}=H'$.
\end{proof}

\begin{proof}[Second proof of Theorem~\ref{cor:The_first_L2-Betti_number_and_F_p-approximation}]~%
\\
\eqref{cor:The_first_L2-Betti_number_and_F_p-approximation:monotone}
This is a direct consequence of the following well-known fact: if $H$ is a normal subgroup of
$p$-power index in $G$, then $b_1(H;\IF_p)-1\leq [G:H] (b_1(G;\IF_p)-1)$
(see, e.g., \cite[Proposition~3.7]{Lackenby(2007covering)}).
\\[2mm]~\eqref{cor:The_first_L2-Betti_number_and_F_p-approximation:inequality}
Choose an epimorphism $\pi\colon F\to G$, where $F$ is a finitely generated free group.
Fix $n\in\IN$, let $F_n=\pi^{-1}(G_n)$ and $H=[F_n,F_n]F_n^p$. Then $H$ is a finite
index subgroup of $F$, so we can choose a presentation $(X,R)$ of $G$ associated with $\pi$
such that $R=R_1\sqcup R_2$, where $R_1$ is finite and $R_2\subseteq H$.

Consider the finitely presented group $\widetilde G=\langle X\mid R_1\rangle$.
We have natural epimorphisms $\phi\colon  \widetilde G\to G$ and $\psi\colon F\to \widetilde G$,
with $\phi\psi=\pi$. If we let $\widetilde G_i=\phi^{-1}(G_i)$ and $\widetilde K=\bigcap_{i=1}^{\infty} \widetilde G_i$,
then $\widetilde G/\widetilde K\cong G$. Thus, applying
Theorem~\ref{thm:Lueck_notes}~\eqref{thm:Lueck_notes:inequalities}
to the group $\widetilde G$ and its subgroups $(\widetilde G_i)$, we get
$b_1^{(2)}(G)\leq \lim_{i \to \infty} \frac{b_1(\widetilde G_i)}{[\widetilde G:\widetilde G_i]}$.
Clearly, $\lim_{i \to \infty} \frac{b_1(\widetilde G_i)}{[\widetilde G:\widetilde G_i]}\leq
\lim_{i \to \infty} \frac{b_1(\widetilde G_i;\IF_p)}{[\widetilde G:\widetilde G_i]}$,
and by assertion~\eqref{cor:The_first_L2-Betti_number_and_F_p-approximation:monotone},
\[
\lim_{i \to \infty} \frac{b_1(\widetilde G_i;\IF_p)}{[\widetilde G:\widetilde G_i]}\leq
\frac{b_1(\widetilde G_n;\IF_p)}{[\widetilde G:\widetilde G_n]}=
\frac{b_1(\widetilde G_n;\IF_p)}{ [G: G_n]}.
\]
Since $G\cong \widetilde G/\langle \langle \psi(R_2)\rangle\rangle$ and by construction
$\psi(R_2)\subseteq \psi(H)=[\widetilde G_n,\widetilde G_n]{\widetilde G_n}^p$,
we have $\ker\phi \subseteq [\tilde G_n,\tilde G_n]\tilde G_n^p$, and therefore
$b_1(\widetilde G_n;\IF_p)=b_1(\phi(\widetilde G_n);\IF_p)=b_1(G_n;\IF_p)$.

Combining these inequalities, we get $b_1^{(2)}(G)\leq \frac{b_1(G_n;\IF_p)}{[G:G_n]}$.
Since $n$ is arbitrary, the proof is complete.
\end{proof}

 \typeout{------------   Section 4: A counterexample with non-trivial intersection ----------------------}

\section{A counterexample with non-trivial intersection}
\label{sec:A_counterexample_with_non-trivial_intersection}
In this section we show that the answer to
Questions~\ref{que:Q-approximation_versus-F_p_approximation}
and~\ref{que:RG_and_limits_of_Betti_numbers} could be negative
for a finitely presented group $G$ and a strictly descending chain $(G_i)_{i\in\IN}$
of normal subgroups of $p$-power index if the intersection
$\cap_{i\in\IN} G_i$ is non-trivial (see inequalities
\eqref{inequality_nontrivialintersection} below).

We start with a finitely generated group $H$ (which will be specified later) and let
$G=H\ast \mathbb Z$. Choose a strictly increasing sequence of positive integers $n_1,n_2,\ldots$ with $n_i\mid n_{i+1}$
for each $i$, and let $G_i \subseteq G$ be the preimage of $n_i \cdot \IZ$
under the natural projection $\pr \colon G= \IZ \ast H \to \IZ$. Then $(G_i)_{i\in\IN}$
is a descending chain of normal subgroups of $G$ with
$\bigcap_{i \ge 1} G_i = \ker(\pr)$.
Let $BG_i \to BG$ be the covering of $BG$ associated
to $G_i \subseteq G$. Then $BG_i$ is homeomorphic to $S^1 \vee \left(\bigvee_{j= 1}^{n_i} BH\right)$.
We have
\[G_i \cong \pi_1(BG_i) \cong \pi_1\left(S^1 \vee \left(\bigvee_{j = 1}^{n_i} BH\right)\right) \cong \IZ \ast (\ast_{j=1}^{n_i} H).\]
Since for any groups $A$ and $B$ we have $A\ast B/[A\ast B,A\ast B]\cong A/[A,A]\oplus B/[B,B]$ 
and $d(A \ast B) = d(A ) + d(B)$ by Grushko-Neumann theorem (see~\cite[Corollary~2 in Section~8.5 on page~227]{Cohen(1989)}, we conclude
\begin{eqnarray*}
H_1(G_i;K) & = & K \oplus \bigoplus_{j=1}^{n_i} H_1(H;K);
\\
H_1(G_i) & = & \IZ\oplus \bigoplus_{j=1}^{n_i} H_1(H);
\\
d(G_i) & = & 1 + n_i \cdot d(H);
\\
\lim_{i \to \infty} \frac{b_1(G_i;K)}{n_i} & = & b_1(H;K);
\\
\lim_{i \to \infty} \frac{d(H_1(G_i))}{n_i} & = & d(H_1(H));
\\
\RG(G;(G_i)_{i \ge 1}) & = & d(H).
\end{eqnarray*}

Now let $p\neq q$ be distinct primes and $H=\mathbb Z/p \mathbb Z\ast \mathbb Z/q \mathbb Z\ast \mathbb Z/q \mathbb Z$.
Clearly we have 
\begin{equation}
\label{eq:0123}
b_1(H) = 0,\quad b_1(H;\IF_p) = 1,\quad d(H_1(H))= 2,\quad  d(H)=3. 
\end{equation}
Hence we obtain
\begin{equation}
\label{inequality_nontrivialintersection}
\lim_{i \to \infty} \frac{b_1(G_i )}{[G:G_i]} < \lim_{i \to \infty}
\frac{b_1(G_i;\IF_p)}{[G:G_i]}
< \lim_{i \to \infty} \frac{d(H_1(G_i))}{[G:G_i]} < \RG(G;(G_i)_{i \ge 1}).
\end{equation}

Using a different $H$ 
we can produce an example of this type where $G$ has a very strong finiteness property, namely, 
$G$ has finite $2$-dimensional $BG$. The construction below is due to Denis Osin
and is simpler and more explicit than the original version of our example.

Again, let $p\ne q$ be two primes. Consider the group
$$
H=\langle x,y,z \mid x^p=u, y^q=v, z^q=w\rangle,
$$
where $u,v,w$ are words from the commutator subgroup of the free group $F$ with basis $x,
y,z$ such that the presentation of $H$ satisfies the $C^\prime (1/6)$ small
cancellation condition. Such words are easy to find explicitly.  
Note that $G=H\ast \mathbb Z$ is a torsion-free $C^\prime (1/6)$ group, hence it has
a finite $2$-dimensional $BG$.

Since $u,v,w\in [F,F]$, we have $b_1(H) = 0$, $b_1(H;\IF_p) = 1$, $d(H_1(H))
= 2$. Further it follows from \cite[Corollary 2]{Erschler(2004randomwalks)} that
the exponential growth rate of $H$ can be made arbitrarily close to $2\cdot
3-1=5$, the exponential growth rate of the free group of rank $3$, by taking
sufficiently long words $u,v,w$. As the exponential growth rate of an
$m$-generated group is bounded from above by $2m-1$, we obtain $d(H) = 3$
whenever $u,v,w$ are sufficiently long. (For details about the exponential
growth rate we refer to \cite{Erschler(2004randomwalks)}.) 

By using a more elaborated construction from~\cite{Wise(1998)}, one can make
such a group $G$ the fundamental group of a compact $2$-dimensional $CAT(-1)$
$CW$-complex.
Other examples of this type can be found in~\cite{Bergeron-Linnell-Lueck-Sauer(2012)}
and~\cite{Lueck(2013l2approxfib)}.

 \typeout{------------   Section 5: $\IQ$-approximation without limit ----------------------}

\section{$\IQ$-approximation without limit}
\label{sec:Q-approximation_without_limit}

In this section we prove the following theorem, which trivially implies
Theorem~\ref{Qappr_nolimit}.

\begin{theorem}
\label{thm:Qapproximation}
Let $d\geq 2$ be a positive integer, let
$p$ be a prime and let
$\varepsilon$ be a real number satisfying $0 < \varepsilon < 1$. Then there exist a  group $G$
with $d$ generators and a descending chain $G=G_{0}\supseteq G_1\supseteq G_2\ldots$ of
normal subgroups of $G$ of $p$-power index with $\bigcap_{i=1}^{\infty} G_i=\{1\}$
with the following properties:
\begin{itemize}
\item[(i)] $\liminf_{i\to\infty}\frac{b_1(G_{2i})}{[G:G_{2i}]}\geq d-1-\varepsilon$;
\item[(ii)] $\lim_{i\to\infty}\frac{b_1(G_{2i-1})}{[G:G_{2i-1}]}=0$.
\end{itemize}
Moreover, if $q$ is a prime different from $p$, we can replace (ii) by a stronger
condition (ii)':
\begin{itemize}
\item[(ii')] $\lim_{i\to\infty}\frac{b_1(G_{2i-1};\IF_q)}{[G:G_{2i-1}]}=0$.
\end{itemize}
\end{theorem}
Note that the last assertion of Theorem~\ref{thm:Qapproximation} shows that the answer
to Question~\ref{que:limit_independent_of_chain} can be negative when ${\rm char}(K)=q>0$ if we do not require that $(G_i)$ is a $q$-chain.


\subsection{Preliminaries}
\label{sec:Preliminaries}
Throughout this section $p$ will be a fixed prime number.
Given a finitely generated group $G$, we will denote
by $G_{\hat p}$ the pro-$p$ completion of $G$ and by
$G_{(p)}$ the image of $G$ in $G_{\hat p}$ (which is isomorphic
to the quotient of $G$ by the intersection of normal
subgroups of $p$-power index). Given a set $X$, by $F(X)$
we denote the free group on $X$.

Let $F$ be a free group and $w\in F$ a non-identity element.
Given $n\in \IN$, denote by $\sqrt[n]{w}$ the unique element
of $F$ whose $n^{\rm th}$ power is equal to $w$ (if such element
exists). Define $e_p(w,F)$ to be the largest
natural number  $e$ with the property that $\sqrt[p^e]{w}$ exists in $F$.

\begin{lemma}
\label{lem:Puchta}
Let $(X,R)$ be a presentation of a group $G$
with $X$ finite, $F=F(X)$ and $\pi\colon F\to G$ the natural projection.
Let $H$ be a normal subgroup of $p$-power index in $G$,
and let $F_H=\pi^{-1}(H)$. Then $H=F_H/\langle \langle R_H \rangle\rangle$
where $R_H$ contains $\frac{[G:H]}{p^{e_p(r,F)-e_p(r,F_H)}}$
$F$-conjugates of $r$ for each $r\in R$ and
no other elements.
\end{lemma}
\begin{proof}
Very similar results are proved in both~\cite{Osin(2011_rankgradient)}
and~\cite{Schlage-Puchta(2012)}, but for completeness we give a proof.
For each $r\in R$, write $r=w(r)^{p^{e_p(r,F)}}$, and choose
a right transversal $T=T(r)$ for $\la w(r)\ra F_H$ in $F$. Then,
since $w(r)$ commutes with $r$,
by  \cite[Lemma~2.3]{Olshanskii-Osin(2008)} we have
$\la r\ra^{F}=\la \{t^{-1}rt: t\in T\}\ra^{F_H}$. Hence $\la \{t^{-1}rt: r\in R, t\in T(R)\}\ra^{F_H}=
\la R\ra^{F}=\ker\pi=\ker(F_H\to H)$, and so it suffices to prove
that $|T(r)|=\frac{[G:H]}{p^{e_p(r,F)-e_p(r,F_H)}}$.

We have
\begin{multline*}
|T(r)|=[F:\la w(r)\ra F_H]=\frac{[F:F_H]}{[\la w(r)\ra F_H:F_H]}=
\frac{[G:H]}{[\la w(r)\ra:\la w(r)\ra \cap F_H]}
\end{multline*}
Finally note that $[\la w(r)\ra:\la w(r)\ra \cap F_H]$ is equal to
$p^k$ for some $k$ (as it divides $[F:F_H]=p^n$), so
$\la w(r)\ra \cap F_H=\la w(r)^{p^k}\ra$.
But then from definition of $e_p(r,F_H)$ we easily
conclude that $((w(r)^{p^k})^{p^{e_p(r,F_H)}}=r=w(r)^{p^{e_p(r,F)}}$.
Hence $k=e_p(r,F)-e_p(r,F_H)$ and $|T(r)|=\frac{[G:H]}{p^{e_p(r,F)-e_p(r,F_H)}}$,
as desired.
\end{proof}

The following definition was introduced by Schlage-Puchta in~\cite{Schlage-Puchta(2012)}.

\begin{definition} \label{def_p-deficiency}
Given a group presentation by generators and relators $(X,R)$,
where $X$ is finite, its \emph{$p$-deficiency} ${\rm def}_p(X,R)\in \IR\cup\{-\infty\}$
is defined by
\[
{\rm def}_p(X,R) = |X|-1- \sum_{r \in R} \;\frac{1}{p^{e_p(r,F(X))}}.
\]
The $p$-deficiency of a finitely generated group $G$ is the supremum
of the set $\{{\rm def}_p(X,R)\}$ where $(X,R)$ ranges over all presentations
of $G$.
\end{definition}

The main motivation for introducing $p$-deficiency in~\cite{Schlage-Puchta(2012)}
was to construct a finitely generated $p$-torsion group with positive rank gradient.
Indeed, it is clear that there exist $p$-torsion groups with positive $p$-deficiency,
and in \cite{Schlage-Puchta(2012)} it is proved that a group with positive $p$-deficiency
has positive rank gradient (in fact, positive $p$-gradient). This is one of the results
indicating that groups of positive $p$-deficiency behave similarly to groups of deficiency
greater than $1$ (all of which trivially have positive $p$-deficiency for any $p$).

Lemma~\ref{lem:pregular} below shows that a finitely presented group $G$ of positive
$p$-deficiency actually contains a normal subgroup of $p$-power index with deficiency greater than $1$,
provided that the presentation of $G$ yielding positive $p$-deficiency is finite and
satisfies certain technical condition.

\begin{definition} A presentation $(X,R)$ of a group $G$
will be called \emph{$p$-regular} if for any $r\in R$ such that
$\sqrt[p]{r}$ exists in $F(X)$, the image of $\sqrt[p]{r}$ in
$G_{(p)}$ is non-trivial. This is equivalent to saying that
if we write each $r\in R$ as $r=v^{p^e}$, where $v$ is not
a $p^{\rm th}$ power in $F(X)$, then the image of $v$
in $G_{(p)}$ has order $p^e$.
\end{definition}

\begin{lemma}
\label{lem:pregular}
Let $(X,R)$ be a finite $p$-regular presentation
of a group $G$. Then there exists a normal subgroup of $p$-power
index $H$ of $G$ with $\frac{{\rm def}(H)-1}{[G:H]}\geq {\rm def}_p(X,R)$.
\end{lemma}
\begin{proof} Let $F=F(X)$. Let $r_1,\ldots, r_m$ be the elements
of $R$ and let $s_i=\sqrt[p]{r_i}$, whenever it is defined in $F(X)$.

Let $\pi\colon F\to  G_{(p)}$ be the natural projection.
Since the presentation $(X,R)$ is $p$-regular,
$\pi(s_i)$ is non-trivial whenever $s_i$ is defined, and since the group
$G_{(p)}$ is residually-$p$,  there exists a normal subgroup $H'$ of
$G_{(p)}$ of $p$-power index  which contains none of the elements $\pi(s_i)$.

Let $F_H=\pi^{-1}(H')$. By construction, $s_i\not\in F_H$, but $r_i\in F_H$,
and therefore $e_p(r_i,F_H)=0$ for each $i$. Let $H$ be the image of $F_H$ in $G$.
Then  by Lemma~\ref{lem:Puchta}, $H$ has a presentation with
$d(F_H)$ generators and $\sum_{i=1}^m\frac{[G:H]}{p^{e_p(r_i,F)}}$ relators.
Since $d(F_H)-1=(|X|-1)[F:F_H]=(|X|-1)[G:H]$ by the Schreier formula,
we get
\[
{\rm def} (H)-1\geq [G:H] \cdot \biggl(|X|-1-\sum_{i=1}^m p^{-e_p(r_i,F)} \biggr) = [G:H]\cdot {\rm def}_p (X,R).
\]
\end{proof}

\begin{lemma}
\label{lem:pregular2} Let $(X,R)$ be a finite $p$-regular presentation,
and let $G=\langle X|R \rangle$. Let $f\in F(X)$ be such that the
image of $f$ in the pro-$p$ completion of $G$ has infinite order.
Then there exists $N\in\IN$ such that for all $n\geq N$
the presentation $(X,R\cup\{f^{p^n}\})$ is $p$-regular.
\end{lemma}
\begin{proof} Let $r_1,\ldots, r_m$ be the elements of $R$.
By assumption there is a normal
subgroup of $p$-power index $H$ of $G$ such that $\sqrt[p]{r_i}$ does not
vanish in $G/H$ (whenever $\sqrt[p]{r_i}$ exists in $F(X)$).
Let $\pi\colon F(X)\to G$ be the natural projection, and choose $N\in\IN$ satisfying
$\pi(f^{p^N})\in H$.

Let $n\geq N$, let $g=\pi(f)$, and let $G'=G/\langle\langle g^{p^n}\rangle \rangle
= \langle X | R\cup\{f^{p^n}\}\rangle$.
We claim that the presentation $(X,R\cup\{f^{p^n}\})$ is $p$-regular.
We need to check that
\begin{itemize}
\item[(i)] each $\sqrt[p]{r_i}$ does not vanish in $G'_{\hat p}$
\item[(ii)] $f^{p^{n-1}}$ does not vanish in $G'_{\hat p}$
\end{itemize}
The kernel of the natural map $G\to G'_{\hat p}$ is contained in $H$
since $g^{p^n}\in H$ and $G/H$ is a finite $p$-group. Since $\pi(\sqrt[p]{r_i})\not\in H$,
this implies (i). Further, an element $x\neq 1$ of a pro-$p$ group
cannot lie in the closed normal subgroup generated by $x^p$.
Hence if $\hat g$ is the image of $g$ (also the image of $f$) in $G_{\hat p}$, then
$\hat g^{p^{n-1}}$ does not lie in the closed normal subgroup of $G_{\hat p}$
generated by $\hat g^{p^{n}}$, call this subgroup $C$. Finally, by definition
of $G'$, there is a canonical isomorphism from $G_{\hat p}/C$ to
$G'_{\hat p}$, which maps the image of $f$ in $G_{\hat p}/C$
to the image of $f$ in $G'_{\hat p}$. Thus, we verified (ii).
\end{proof}

\begin{corollary}
\label{cor:pregular}
Let $(X,R)$ be a finite $p$-regular presentation,
and let $G=\langle X\mid R \rangle$. Let $H\subseteq K$ be normal subgroups of $F(X)$
of $p$-power index, and let $\delta>0$ be a real number.
Then there exists a finite set $R'\subset [K,K]$ with
 $\sum\limits_{r\in R'}p^{-e_p(r,F(X))}<\delta$
such that
\begin{enumerate}
\item \label{cor:pregular:regular}  the presentation $(X,R\cup R')$ is $p$-regular;
\item \label{cor:pregular:inequality} if $G'=\langle X\mid R\cup R'\rangle$ and $H'$ is the image of $H$ in $G'$,
then $b_1(H')\leq d(K)$.
\end{enumerate}
Moreover, if $q$ is a prime different from $p$, we can require that
$b_1(H';\IF_q)\leq d(K)$.
\end{corollary}
\begin{proof}
If $b_1(H;\IF_q) \le d(K)$, we can choose $R' = \emptyset$.
Hence we can assume without loss of generality that $b_1(H;\IF_q)>d(K)$.
Clearly, it suffices to prove a weaker statement,
where inequality $b_1(H';\IF_q)\leq d(K)$ is replaced by
$b_1(H';\IF_q) <b_1(H;\IF_q)$.
The assertion of Corollary~\ref{cor:pregular} then follows by
repeated applications with $\delta$ replaced by $\delta/(b_1(H,\IF_q) - d(K))$.

Let $Y$ be any free generating set for $H$. Obviously $K/[K,K]$ is a free
abelian group of rank $d(K)$. Any (finite) matrix over the integers can be
transformed by elementary row and column operations to a diagonal matrix. Hence
by applying elementary transformations to $Y$, we can arrange that $Y$ is a
disjoint union $Y_1\sqcup Y_2$ where $|Y_1|\leq d(K)$ and $Y_2\subseteq [K,K]$.

Let $L=\langle Y_2\rangle$, the subgroup generated by $Y_2$. Since
$b_1(H;\IF_q)>d(K)$,
there exists $f\in Y_2$ whose image in $H/[H,H]H^q\cong H_1(H,\IF_q)$ is non-trivial.
Now apply Lemma~\ref{lem:pregular2} to this $f$, choose $n$
such that $\frac{1}{p^n}<\delta$ and let $R'=\{f^{p^n}\}$.
The choice of $f$ ensures that $b_1(H';\IF_q)<b_1(H;\IF_q)$, so $R'$ has the required properties.
\end{proof}


\subsection{Proof of Theorem~\ref{thm:Qapproximation}}
\label{subsec:Qapproximation}
To simplify the notations, we will give a proof of the main part
of Theorem~\ref{thm:Qapproximation}. The last part of Theorem~\ref{thm:Qapproximation}
is proved in the same way by using the last assertion of Corollary~\ref{cor:pregular}.

We start by giving an outline of the construction.
Let $F = F(X) $ be a free group of rank $d = |X|$. Below we shall define a descending chain
$F=F_0\supseteq F_1\supseteq \ldots$ of normal subgroups of $F$ of $p$-power
index and a sequence of finite subsets $R_1, R_2,\ldots$ of $F$. Let
$R=\bigcup_{i=1}^{\infty}R_n$.
For each $n\in\IZ_{\geq 0}$ we let
$G(n)= F/\langle\langle\bigcup_{i=1}^n R_i\rangle\rangle$, $G(\infty)
=
\varinjlim G(i)
=
F/\langle\langle  R\rangle\rangle$ and let $G$ be the image of $G(\infty)$ in its pro-$p$ completion.
Denote by $G(n)_i$, $G(\infty)_i$ and $G_i$ the canonical image of $F_i$ in
$G(n),G(\infty)$ and $G$, respectively.
We will show that the group $G$ and its subgroups $(G_i)$
satisfy the conclusion of Theorem~\ref{thm:Qapproximation}.

Fix a sequence of positive real numbers $(\delta_n)$ which converges to zero
and a descending chain $(\Phi_n)$ of normal subgroups of $p$-power index in $F$ which
form a base of neighborhoods of $1$ for the pro-$p$ topology.
The subgroups $F_n$ and relator sets $R_n$ will be constructed inductively
so that the following properties hold:

\begin{itemize}
\item[(i)]For $n\geq 0$ we have
\[
\frac{b_1(G(n)_{2n})}{[G(n):G(n)_{2n}]} > d-1-\varepsilon;
\]
\item[(ii)] For $n\geq 1$ we have
\[
\frac{b_1(G(n)_{2n-1})}{[G(n):G(n)_{2n-1}]}<\delta_n;
\]
\item[(iii)] $R_{n}$ is contained in $[F_{2n-2},F_{2n-2}]$ for $n\geq 1$;
\item[(iv)] $F_{2n}\subseteq \Phi_{n}$ for $n\geq 1$;
\item[(v)] ${\rm def}_p(X,\cup_{i=1}^n R_i)>d-1-\varepsilon $ for $n\geq 1$;
\item[(vi)] The presentation $(X,\cup_{i=1}^n R_i)$ is $p$-regular for $n\geq 1$.
\end{itemize}

We first explain why properties (i)-(vi) will imply that the group $G$
and its subgroups $(G_n)$ have the desired properties.
Each $G_n$ is normal of $p$-power index in $G$ since
$F_n$ is normal of $p$-power index in $F$.
Condition (iv) implies that $(G_n)$ is a base of neighborhoods of $1$
for the pro-$p$ topology on $G$, and since $G$ is residually-$p$ by construction,
we have $\bigcap_{n = 1}^{\infty} G_n=\{1\}$.

Condition (iii) implies that
$[G(n):G(n)_{i}]=[G(\infty):G(\infty)_{i}]$ and
$b_1(G(n)_i)=b_1(G(\infty)_i)$ for $i\leq 2n$. Since
$G(\infty)_i$ is normal of $p$-power index in $G(\infty)$,
the group $G(\infty)/[G(\infty)_i,G(\infty)_i]$ is residually-$p$,
so both the index and the first Betti number of $G(\infty)_i$ do
not change under passage to the image in the pro-$p$ completion of $G(\infty)$:
$[G:G_i]=[G(\infty):G(\infty)_i]$ and $b_1(G_i)=b_1(G(\infty)_i)$.
In view of these equalities, conditions (i) and (ii) yield
the corresponding conditions in Theorem~\ref{thm:Qapproximation}.

We now describe the construction of the sets $R_n$ and subgroups $F_n$.
The base case $n=0$ is obvious: we set $F_0=F$ and $G(0)=F$, and the only
condition we require for $n=0$ (condition (i)) clearly holds.

Suppose now that $N\in\IN$ and we constructed subsets $(R_i)_{i=1}^N$
and subgroups $(F_i)_{i=1}^{2N}$ such that (i)-(vi)
hold for all $n\leq N$.

Let $F_{2N+1}=[F_{2N},F_{2N}]F_{2N}^{p^e}$ where $e$ is specified below.
Then $F_{2N+1}$ is a normal subgroup of $p$-power index in $F$
and $F_{2N}\supseteq F_{2N+1}\supset [F_{2N},F_{2N}]$.
Since $b_1(G(N)_{2N})>0$ by $(i)$ for $n=N$ and hence
\begin{eqnarray*}
p^e
&\le &
\bigl|H_1(G(N)_{2N})/p^e \cdot H_1(G(N)_{2N})\bigr|
\\
& = &
\bigl|G(N)_{2N}/ [G(N)_{2N},G(N)_{2N}]G(N)_{2N}^{p^e}\bigr|
\\
& = & |G(N)_{2N}/ G(N)_{2N+1}|
\\
& = &
 [G(N)_{2N}:G(N)_{2N+1}]
\\
& \le &
[G(N):G(N)_{2N+1}],
\end{eqnarray*}
so we can arrange
\[
\frac{d(F_{2N})}{[G(N):G(N)_{2N+1}]} < \delta_{N+1}
\]
by choosing $e$ large enough.

Now applying Corollary~\ref{cor:pregular} with $H=F_{2N+1}$, $K=F_{2N}$ and
$\delta={\rm def}_p(X,\cup_{i=1}^N R_i)-(d-1-\varepsilon)$,
we get that there is a finite subset $R_{N+1}\subseteq [F_{2N},F_{2N}]$ such that
the presentation $(X,\cup_{i=1}^{N+1} R_i)$ is $p$-regular and
${\rm def}_p(X,\cup_{i=1}^{N+1} R_i) > d-1-\varepsilon $.
Hence conditions (iii),(v),(vi) hold for $n=N+1$.
The subgroup $H'$ in the notations of Corollary~\ref{cor:pregular}
is equal to $G(N+1)_{2N+1}$, so $b_1(G(N+1)_{2N+1}) \le  d(F_{2N})$.
Since condition (iii) implies $[G(N+1):G(N+1)_{2N+1}] = [G(N):G(N)_{2N+1}]$,
we conclude
\[\frac{b_1(G(N+1)_{2N+1})}{[G(N+1):G(N+1)_{2N+1}]} \le \frac{d(F_{2N})}{[G(N):G(N)_{2N+1}]}< \delta_{N+1}.
\]
Thus we have shown that conditions (ii),(iii),(v),(vi) hold for $n=N+1$.

It remains to construct $F_{2N+2}$ and to verify (i) and (iv) for $n = N+1$.
We apply Lemma~\ref{lem:pregular} to $G(N+1) = \langle X \mid \cup_{i=1}^{N+1} R_i\rangle$
and obtain using (v) a normal subgroup $H$ of $G(N+1)$ of $p$-power index satisfying
\[
\frac{{\rm def}(H)-1}{[G(N+1):H]}> d -1- \varepsilon.
\]
Let $F_{2N+2}\subseteq F_{2N+1}\cap \Phi_{N+1}$ be the intersection of the preimage of $H$ under the projection
$p_{N+1} \colon F_{N+1} \to G(N+1)$ with $F_{2N+1} \cap \Phi_{N+1}$. Obviously (iv) for holds $n = N+1$.
Then $G(N+1)_{2N+2}$ is a subgroup of
$H$ of finite index.  The quantity ${\rm def}(\cdot)-1$ is supermultiplicative, i.e., if $L$ is a finite index subgroup of $H$,
then ${\rm def}(L)-1\geq [H:L]\cdot ({\rm def}(H)-1)$, see for instance~\cite[Lemma~2.2]{Osin(2011_rankgradient)}. Hence we conclude
\[
\frac{{\rm def}(G(N+1)_{2N+2})-1}{[G(N+1):G(N+1)_{2N+2})]} \geq \frac{{\rm def}(H)-1}{[G(N+1):H]}> d -1- \varepsilon.
\]
Since $b_1(G(N+1)_{2N+2}) \ge {\rm def}(G(N+1)_{2N+2})$, condition (i) holds for $n = N+1$. This finishes the proof of
Theorem~\ref{thm:Qapproximation}.

\typeout{-------------------------------------- References  ---------------------------------------}


\begin{thebibliography}{10}

\bibitem{Abert-Jaikin-Zapirain-Nikolov(2011)}
M.~Ab{\'e}rt, A.~Jaikin-Zapirain, and N.~Nikolov.
\newblock The rank gradient from a combinatorial viewpoint.
\newblock {\em Groups Geom. Dyn.}, 5(2):213--230, 2011.

\bibitem{Abert-Nikolov(2012)}
M.~Ab{\'e}rt and N.~Nikolov.
\newblock Rank gradient, cost of groups and the rank versus {H}eegaard genus
  problem.
\newblock {\em J. Eur. Math. Soc. (JEMS)}, 14(5):1657--1677, 2012.

\bibitem{Bergeron-Linnell-Lueck-Sauer(2012)}
N.~Bergeron, P.~Linnell, W.~L\"uck, and R.~Sauer.
\newblock On the growth of {B}etti numbers in $p$-adic analytic towers.
\newblock Preprint, arXiv:1204.3298v1 [math.GT], to appear in Groups, Geometry,
  and Dynamics, 2012.

\bibitem{Cohen(1989)}
D.~E. Cohen.
\newblock {\em Combinatorial group theory: a topological approach}.
\newblock Cambridge University Press, Cambridge, 1989.

\bibitem{Elek(2003c)}
G.~Elek.
\newblock The rank of finitely generated modules over group algebras.
\newblock {\em Proc. Amer. Math. Soc.}, 131(11):3477--3485 (electronic), 2003.

\bibitem{Erschler(2004randomwalks)}
A.~Erschler.
\newblock Growth rates of small cancellation groups.
\newblock In {\em Random walks and geometry}, pages 421--430. Walter de Gruyter
  GmbH \& Co. KG, Berlin, 2004.

\bibitem{Gaboriau(2000b)}
D.~Gaboriau.
\newblock Co\^ut des relations d'\'equivalence et des groupes.
\newblock {\em Invent. Math.}, 139(1):41--98, 2000.

\bibitem{Gaboriau(2002a)}
D.~Gaboriau.
\newblock Invariants {$l\sp 2$} de relations d'\'equivalence et de groupes.
\newblock {\em Publ. Math. Inst. Hautes \'Etudes Sci.}, 95:93--150, 2002.

\bibitem{Gaboriau(2002b)}
D.~Gaboriau.
\newblock On orbit equivalence of measure preserving actions.
\newblock In {\em Rigidity in dynamics and geometry (Cambridge, 2000)}, pages
  167--186. Springer, Berlin, 2002.

\bibitem{Lackenby(2005expanders)}
M.~Lackenby.
\newblock Expanders, rank and graphs of groups.
\newblock {\em Israel J. Math.}, 146:357--370, 2005.

\bibitem{Lackenby(2007covering)}
M.~Lackenby.
\newblock Covering spaces of 3-orbifolds.
\newblock {\em Duke Math. J.}, 136(1):181--203, 2007.

\bibitem{Linnell-Lueck-Sauer(2011)}
P.~Linnell, W.~L{\"u}ck, and R.~Sauer.
\newblock The limit of {$\Bbb F_p$}-{B}etti numbers of a tower of finite covers
  with amenable fundamental groups.
\newblock {\em Proc. Amer. Math. Soc.}, 139(2):421--434, 2011.

\bibitem{Lueck(1994c)}
W.~L{\"u}ck.
\newblock Approximating ${L}\sp 2$-invariants by their finite-dimensional
  analogues.
\newblock {\em Geom. Funct. Anal.}, 4(4):455--481, 1994.

\bibitem{Lueck(2002)}
W.~L{\"u}ck.
\newblock {\em {$L\sp 2$}-{I}nvariants: {T}heory and {A}pplications to
  {G}eometry and \mbox{{$K$}-{T}heory}}, volume~44 of {\em Ergebnisse der
  Mathematik und ihrer Grenzgebiete. 3.~Folge. A Series of Modern Surveys in
  Mathematics [Results in Mathematics and Related Areas. 3rd Series. A Series
  of Modern Surveys in Mathematics]}.
\newblock Springer-Verlag, Berlin, 2002.

\bibitem{Lueck(2013l2approxfib)}
W.~L{\"u}ck.
\newblock Approximating {$L^2$}-invariants and homology growth.
\newblock {\em Geom. Funct. Anal.}, 23(2):622--663, 2013.

\bibitem{Lueck-Osin(2011)}
W.~L{\"u}ck and D.~Osin.
\newblock Approximating the first {$L^2$}-{B}etti number of residually finite
  groups.
\newblock {\em J. Topol. Anal.}, 3(2):153--160, 2011.

\bibitem{Olshanskii-Osin(2008)}
A.~Y. Olshanskii and D.~V. Osin.
\newblock Large groups and their periodic quotients.
\newblock {\em Proc. Amer. Math. Soc.}, 136(3):753--759, 2008.

\bibitem{Osin(2011_rankgradient)}
D.~Osin.
\newblock Rank gradient and torsion groups.
\newblock {\em Bull. Lond. Math. Soc.}, 43(1):10--16, 2011.

\bibitem{Osin-Thom(2013)}
D.~Osin and A.~Thom.
\newblock Normal generation and {$\ell^2$}-{B}etti numbers of groups.
\newblock {\em Math. Ann.}, 355(4):1331--1347, 2013.

\bibitem{Schlage-Puchta(2012)}
J.-C. Schlage-Puchta.
\newblock A {$p$}-group with positive rank gradient.
\newblock {\em J. Group Theory}, 15(2):261--270, 2012.

\bibitem{Wise(1998)}
D.~T. Wise.
\newblock Incoherent negatively curved groups.
\newblock {\em Proc. Amer. Math. Soc.}, 126(4):957--964, 1998.

\end{thebibliography}

\end{document}